\documentclass[11pt]{article}

\usepackage{amsmath}
\usepackage{amssymb}
\usepackage{amsthm}
\usepackage{multirow}
\usepackage{color}
\usepackage{longtable}
\usepackage{array}
\usepackage{url}
\usepackage{tikz}
\renewcommand{\arraystretch}{1.5}

\newtheorem{theorem}{Theorem}[section]

\newtheorem{lemma}[theorem]{Lemma}
\newtheorem{example}[theorem]{Example}

\newtheorem{corollary}[theorem]{Corollary}

\newtheorem{remark}{Remark}[section]

\title{Notes on Various Methods for Constructing Directed Strongly Regular Graphs}

\author{
Jerod Michel\thanks{Corresponding author. Email Address: contextolibre@gmail.com.}
\thanks{J. Michel is with the Department of Mathematics, Zhejiang University, Hangzhou 310027, China.}
and Baokun Ding\thanks{B. Ding is with the Department of Mathematics, Zhejiang University, Hangzhou 310027, China.} \\
}

\begin{document}

\date{}\maketitle

\begin{abstract}
Duval, in ``A Directed Graph Version of Strongly Regular Graphs'' [{\it Journal of Combinatorial Theory}, Series A 47 (1988) 71 - 100], introduced the concept of directed strongly regular graphs. In this paper we construct several rich families of directed strongly regular graphs with new parameters. Our constructions yielding new parameters are based on extending known explicit constructions to cover more parameter sets. We also explore some of the links between Cayley graphs, block matrices and directed strongly regular graphs with certain parameters. Directed strongly regular graphs which are also Cayley graphs are interesting due to their having more algebraic structure. We construct directed strongly regular Cayley graphs with parameters $((m+1)s,ls,ld,ld-d,ld)$ where $d,l$ and $s$ are integers with $dm=ls$ and $1\leq l<m$. We also give a new block matrix characterization for directed strongly regular graphs with parameters $(m(dm+1),dm,m,m-1,m)$, which were first dicussed by Duval et al. in ``Semidirect Product Constructions of Directed Strongly Regular Graphs'' [{\it Journal of Combinatorial Theory}, Series A 104 (2003) 157 - 167].

\smallskip

\end{abstract}


\section{Introduction}\label{sec1}
R. C. Bose, in the early 1960s, introduced the concept of strongly regular graphs \cite{BOSE}. A. M. Duval, in 1988, introduced the concept of directed strongly regular graphs \cite{DUV2}. Development of the theory of strongly regular graphs was motivated by the study of finite permutation groups, classification of finite simple groups, as well as association schemes. Strongly regular graphs arise from such geometric and combinatorial objects as finite fields, finite geometries, combinatorial designs, and algebraic codes \cite{BOSE}, \cite{CAL} and \cite{FX}. The sources of directed strongly regular graphs have been reported on in \cite{DUV2} and \cite{DUV1}, and, more recently, in \cite{AGOS}, \cite{BOS}, \cite{SG}, \cite{OLMEZ} and \cite{OS}. Duval, in \cite{DUV2}, provided an initial set of parameter restrictions for directed strongly regular graphs, and a complete list of feasible parameters can be found in \cite{BROU}.
\par
This paper investigates various methods for constructing directed strongly regular graphs. We discuss explicit constructions, constructions via Cayley graphs, and by using block matrices as well. Section 2 introduces necessary notation and defines strongly regular graphs, directed strongly regular graphs, and certain necessary types of algebraic structures. Section 3 describes new explicit constructions obtained by extending existing construction methods to cover more parameters sets. Section 4 describes a new construction method for directed strongly regular Cayley graphs with parameters $((m+1)s,ls,ld,ld-d,ld)$ where $d,l$ and $s$ are integers with $dm=ls$ and $1\leq l<m$ \cite{OS}. Section 5 gives a new characterization for those directed strongly regular graphs with parameters $(m(dm+1),dm,m,m-1,m)$ \cite{DUV1}. Section 6 concludes the paper and discusses directions for further work.

\section{Preliminaries}\label{sec2}

We will assume some familiarity with group theory. All of the constructions in this correspondence are done over groups. Since we will make reference to loops more than once, we include its definition here.
\par
A {\it loop} \cite{MOOR} is a set $L$ with a binary operation `*' such that \begin{enumerate}
\item there exists a two-sided identity $Id \in L$ where $Id * x = x * Id = x$ for all $x \in L$, and
\item for all $a \in L$ the left-multiplication map $x \mapsto a * x$ is bijective; also the right-multiplication map $x \mapsto x * a$ is bijective.\end{enumerate} The Cayley table of a finite loop with members $1,...,n$ and identity $1$, is simply an $n \times n$ table with first row and column $(1,...,n)$ (in that order), and where every row and column must contain every element exactly once. Note that an associative loop is the same thing as a group. Unless it is clear by convention, we will denote the identity of a group $G$ by $Id_{G}$, or simply by $Id$ if $G$ is the only group in question.

\subsection{Strongly Regular Graphs and Digraphs}

Only finite simple graphs will be considered in this paper, i.e., only graphs with no loops, or multiple edges. Denote the set of vertices resp. edges of a directed graph $\Gamma$ by $V(\Gamma)$ resp. $E(\Gamma)$. For two vertices $x,y \in V(\Gamma)$, we say $x$ is {\it adjacent} to $y$, written $x \rightarrow y$, if there is an edge from $x$ to $y$. If there is also an edge from $y$ to $x$ then we say that there is an {\it undirected edge} between $x$ and $y$, denoted $x \leftrightarrow y$. If neither $x \rightarrow y$ nor $x \leftrightarrow y$ holds, then we say $x$ is not adjacent to $y$, and write $x \nrightarrow y$.
\par
Now let $\Gamma$ be a graph with vertex set $V(\Gamma)=\{v_{i}\}_{i=1}^{n}$. The adjacency matrix $A = A(\Gamma)$ of $\Gamma$ is the $n \times n$ matrix with rows and columns indexed by $V(\Gamma)$ and \[ A_{ij} = \begin{cases} 1 \text{ if } v_{i} \text{ is adjacent to } v_{j}, \\
                                                                                  0 \text{ otherwise }. \end{cases} \]
A {\it strongly regular graph} \cite{BOSE} with parameters $(v,k,\lambda,\mu)$ is an undirected graph with $v$ vertices with adjacency matrix $A$ satisfying \[ A^{2}=kI+\lambda A + \mu (J-I-A), \] and \[ AJ=JA=kJ,\] where I and J are the identity and all-ones matrix respectively. The equations implies that the number of paths of length two from a vertex $x$ to another vertex $y$ is $\lambda$ if $x$ and $y$ are adjacent, $\mu$ if not. The second equation implies that the number of neighbors, or {\it valency}, of any vertex is $k$.
\par
The concept of `strong regularity' generalizes to directed graphs in the following way. A directed graph $\Gamma$ is {\it directed strongly regular} \cite{DUV2} with parameters $(v,k,s;\alpha,\beta)$ if it has adjacency matrix $A$ satisfying \[ A^{2}=tI+\alpha A + \beta (J-I-A), \] and \[ AJ=JA=kJ.\] From these equations we see that each vertex has $k$ in-neighbors and $k$ out-neighbors, $s$ of which are both in- and out-neighbors. For two distinct vertices $x$ and $y$, the number of directed paths of length two from $x$ to $y$ is $\alpha$ if $x \rightarrow y$ and $\beta$ if $x \nrightarrow y$.

\subsection{Cyclotomic Classes and Cyclotomic Numbers}
Cyclotomic classes have proven to be a powerful tool for constructing difference sets, designs, and graphs e.g. see \cite{DHM}, \cite{DING}, \cite{DUV2}, \cite{NOW}. Let $q$ be a prime power, $\mathbb{F}_{q}$ a finite field, and $e$ a divisor of $q-1$. For a primitive element $\alpha$ of $\mathbb{F}_{q}$ let $D_{0}^{e}$ denote $\langle \alpha^{e} \rangle$, the multiplicative group generated by $\alpha^{e}$, and let \[ D_{i}^{e} = \alpha^{i}D_{0}^{e}, \text{ for } i=1,2,...,e-1. \] We call $D_{i}^{e}$ the {\it cyclotomic classes} of order $e$. The {\it cyclotomic numbers} of order $e$ are defined to be \[ (i,j)_{e} = \left| D_{i}^{e} \cap (D_{j}^{e} + 1) \right|. \]

It is easy to see there are at most $e^{2}$ different cyclotomic numbers of order $e$. When it is clear from the context, we simply denote $(i,j)_{e}$ by $(i,j)$.
The cyclotomic numbers $(h,k)$ of order $e$ have the following properties (\cite{D}):

\begin{eqnarray}\label{eq7}
(h,k) & = & (e-h,k-h), \\
(h,k) & = & \begin{cases}
(k,h),                         & \text{if } f \text{ even},\\
(k+\frac{e}{2},h+\frac{e}{2}), & \text{if } f \text{ odd}.
\end{cases}
\end{eqnarray}

\section{ Explicit Constructions of Directed Strongly Regular Graphs from Products of Abelian Groups}
In this section we give an explicit construction which was motivated by Stefan Gyurki's monograph ``New Rich Families of Directed Strongly Regular Graphs'' \cite{SG}, in which it was discussed how to construct directed strongly regular graphs with parameters $(2n^{2},3n-2,2n-1,n-1,3),(3n^{2},4n-2,2n,n,4),\\(2n^{2},4n-2,2n+2,n+2,6)$ and $(3n^{2},6n-2,2n+6,n+6,10)$ from certain loops and quasigroups. We give a construction of directed strongly regular graphs that have parameters \\$(mn^{2},mn+n-2,2n+m-3,n+m-3,m+1)$, where $n\geq 2$ is an integer and $m\geq 3$ is an odd integer. The proof is tedious but not difficult, so we eliminate similar cases.
\begin{theorem}\label{th3.0} (Construction I) Let $G$ be a multiplicative group of order $n\geq 2$. Let $\Gamma$ be the digraph with vertex set $\mathcal{V}=G\times G\times \mathbb{Z}_{m}$ where $m\geq 3$ is odd, and adjacency given by $(x,y,i)\rightarrow (u,v,j)$ if and only if one of the following hold:\begin{enumerate}
\item $x=u,i=j$ and $y\ne v$,
\item $y=v,i=j$ and $x\ne u$,
\item $u=xy$ and $j\in \{i+1,...,i+\frac{m-1}{2}\}$,
\item $v=xy$ and $j\in \{i-\frac{m-1}{2},...,i-1\}$.\end{enumerate} Then $\Gamma$ is a directed strongly regular graph with parameters $(v,k,t,\lambda,\mu)=(mn^{2},mn+n-2,2n+m-3,n+m-3,m+1)$.
\end{theorem}
\begin{proof} Clearly $v=mn^{2}$.
\\[1\baselineskip]
$\underline{k}$: Let $(x,y,i)\in \mathcal{V}$. There are $n-1$ vertices $(z,y,i)\in\mathcal{V}$ with $z\ne x$, and $n-1$ vertices $(x,z,i)\in \mathcal{V}$ with $z\ne y$. Also, there are $n$ vertices $(xy,z,i+\alpha)\in \mathcal{V}$ where $\alpha$ is fixed in $\{1,...,\frac{m-1}{2}\}\subset \mathbb{Z}_{m}$, and $n$ vertices $(z,xy,i)\in \mathcal{V}$ where $\alpha$ is fixed in $\{1,...,\frac{m-1}{2}\}$. Thus $k=2(n-1)+n(m-1)=mn+n-2$.
\\[1\baselineskip]
 $\underline{t}$: We need to count the number of vertices $(u,v,j)\in \mathcal{V}$ such that \[(x,y,i)\leftrightarrow (u,v,j).\] If $x=u$ and $i=j$ but $y\ne v$ then the property holds for $(u,v,j)$, and if $y=v$ and $i=j$ but $x\ne u$ then the property holds for $(u,v,j)$. This gives $2(n-1)$ vertices. Suppose that $(x,y,i)\leftrightarrow (u,v,j)=(xy,z,i+\alpha)$ where $\alpha$ is fixed in $\{1,...,\frac{m-1}{2}\}$. Then we must have $x=zxy$ so that $z=x^{-1}$. Thus there is one such vertex with this property. Now suppose that \[(x,y,i)\leftrightarrow (u,v,j)=(z,xy,i+\alpha)\] where $\alpha$ is fixed in $\{1,...,\frac{m-1}{2}\}$. Then we must have $y=zxy$ so that $z=y^{-1}$. There is one such vertex with this property. Counting all such vertices (and letting $\alpha$ run over $\{1,...,\frac{m-1}{2}\}$), we get $t=2(n-1)+m-1=2n+m-3$.
\\[1\baselineskip]
$\underline{\lambda}$: Now, given two vertices $(x,y,i),(u,v,j)\in \mathcal{V}$ such that $(x,y,i)\rightarrow(u,v,j)$, we need to count the number of vertices $(z,w,l)\in \mathcal{V}$ such that \[(x,y,i)\rightarrow(z,w,l)\text{ and }(z,w,l)\rightarrow(u,v,j).\] The four possibilities for such a vertex $(u,v,j)$ are given in the statement of the theorem. We first consider case $1$, where $u=x$ and $i=j$ but $y\ne v$, which is similar to case $2$; and then consider case $3$, where $u=xy$ and $j\in \{j+1,...,i+\frac{m-1}{2}\}$, which is similar to case $4$. So, assuming that $u=x$ and $i=j$ but $y\ne v$, there are $n-2$ vertices $(z,w,l)\in \mathcal{V}$ such that \[(x,y,i)\rightarrow(z,w,l)=(x,w,i)\] since there are $n-2$ points $w\in G$ such that $y\ne w\ne v$. There are $\frac{m-1}{2}$ vertices $(z,w,l)\in \mathcal{V}$ were $l=i+\alpha$ for some $\alpha\in\{1,...,\frac{m-1}{2}\}$. If \[(x,y,i)\rightarrow(z,w,l)=(z,w,i+\alpha)\] then we must have $z=xy$, and if \[(z,w,l)=(z,w,i+\alpha)\rightarrow(u,v,j)\] then we must have $v=xyw$ so that $w=v(xy)^{-1}$. Also, there are $\frac{m-1}{2}$ vertices $(z,w,l)\in \mathcal{V}$ where $l=i-\alpha$ for some $\alpha\in\{1,...,\frac{m-1}{2}\}$. If $(x,y,i)\rightarrow(z,w,l)=(z,w,i-\alpha)$ then we must have $w=xy$, and if $(z,w,l)=(z,w,i-\alpha)\rightarrow(u,v,j)$ then we must have $x=u=xyw$ so that $w=y^{-1}$. Counting all of the above vertices we have $n-2+m-1=n+m-3$ in total.

 We now consider case $3$, where $u=xy$ and $j\in \{j+1,...,i+\frac{m-1}{2}\}$, i.e., we have \[(x,y,i)\rightarrow(u,v,j)=(xy,v,i+\alpha).\] There are $m$ possibilities for a vertex $(z,w,l)\in \mathcal{V}$ such that $(x,y,i)\rightarrow(z,w,l)$. One case is when $l=i$, $\frac{m-1}{2}$ cases are when $l=i+\beta$ for some $\beta \in \{1,...,\frac{m-1}{2}\}$, and $\frac{m-1}{2}$ cases are when $l=i-\beta$ for some $\beta\in \{1,...,\frac{m-1}{2}\}$. If $l=i$, there are $n-1$ vertices $(z,w,l)\in \mathcal{V}$ with $w=y$ and $z=x$. But then if $(z,w,l)\rightarrow(xy,v,i+\alpha)$ we must have $xy=yz$ so that $z=x$, a contradiction. Also there are $n-1$ vertices $(z,w,l)\in \mathcal{V}$ with $w\ne y$ and $z=x$. But then if $(z,w,l)\rightarrow(xy,v,i+\alpha)$ we must have $xy=zw$ so that $w=y$, a contradiction. Thus, we cannot have $l=i$. Then suppose $l=i+\alpha$ (i.e. $\beta=\alpha$). If $(x,y,i)\rightarrow(z,w,l)$ then we must have $z=xy$, and if $(z,w,l)\rightarrow(xy,v,i+\alpha)$ we must have $w\ne v$. There are $n-1$ such vertices. Now suppose $l=i+\beta$ where $\beta$ occurs before $\alpha$ in the natural ordering of $\{1,...,\frac{m-1}{2}\}$. If $(x,y,i)\rightarrow(z,w,l)$ then again we must have $z=xy$, and if $(z,w,l)\rightarrow(xy,v,i+\alpha)$ then we must have $xy=zw=xyw$ so that $w=1$. There is one such vertex per choice of $\beta$. If we suppose that $l=i+\beta$ where $\beta$ occurs after $\alpha$ in the natural ordering of $\{1,...,\frac{m-1}{2}\}$, then \[(x,y,i)\rightarrow(z,w,l)\] implies that $z=xy$, and \[(z,w,l)\rightarrow(xy,v,i+\alpha)\] implies that $v=xyw$ so that $w=v(xy)^{-1}$. There is one such vertex per choice of $\beta$. Lastly, we suppose that $l=i-\beta$ for some $\beta \in \{1,...,\frac{m-1}{2}\}$. If $(x,y,i)\rightarrow(z,w,l)$ then we must have $w=xy$. If $(z,w,l)\rightarrow(xy,v,i+\alpha)$ then we must have $v=xyw$ so that $w=v(xy)^{-1}$ whenever $\alpha+\beta$ occurs in $\{\frac{m-1}{2}+1,...,m-1\}$, and we must have $xy=zw=zxy$ so that $z=1$ whenever $\alpha+\beta$ occurs in $\{1,..,\frac{m-1}{2}\}$. Either way, there is one such vertex per choice of $\beta$. Counting over all such vertices we have $n-1+m-2=n+m-3$ in total. Thus we can conclude that $\lambda=n+m-3$.
\\[1\baselineskip]
$\underline{\mu}$: Now, given two vertices $(x,y,i),(u,v,j)\in \mathcal{V}$ such that \[(x,y,i)\nrightarrow(u,v,j),\] we need to count the number of vertices $(z,w,l)\in \mathcal{V}$ such that \[(x,y,i)\rightarrow(z,w,l)\text{ and }(z,w,l)\rightarrow(u,v,j).\] There are only two possibilities for $(u,v,j)$. One case is when $u\ne xy$ and $j=i+\alpha$ for some $\alpha\in \{1,...,\frac{m-1}{2}\}$, and one case is when $v\ne xy$ and $j=i-\alpha$ for some $\alpha\in\{1,...,\frac{m-1}{2}\}$. We show only the first case, as the second case is similar. So we have $(x,y,i)\nrightarrow (u,v,i+\alpha)$ for some $\alpha\in\{1,...,\frac{m-1}{2}\}$ and $u\ne xy$. Again there are $m$ possibilities for the vertex $(z,w,l)\in\mathcal{V}$. One case is when $l=i$, $\frac{m-1}{2}$ cases are when $l=i+\beta$ for some $\beta\in\{1,...,\frac{m-1}{2}\}$, and $\frac{m-1}{2}$ cases are when $l=i-\beta$ for some $\beta \in \{1,...,\frac{m-1}{2}\}$. Suppose $l=i$. If $(x,y,i)\rightarrow(z,w,l)$ then we must have $x=z$ and $y\ne w$, or, $x\ne z$ and $y=w$. If $(z,w,l)\rightarrow(u,v,i+\alpha)$ then we must have $w=ux^{-1}$ if $y\ne w$, and $z=uy^{-1}$ if $w=y$. Thus there are two such vertices. Suppose $l=i+\alpha$ (i.e. $\beta=\alpha$). If $(z,w,l)\rightarrow(u,v,i+\alpha)$ then either $z=u$ and $w\ne v$, or $z\ne u$ and $w=v$. If $(x,y,i)\rightarrow(z,w,l)$ we must have $z=xy$, but since $u\ne xy$, we cannot have $z=u$. Thus we must have $w=v$ and $z=xy\ne u$, and there is only one such vertex. If we suppose $l=i+\beta$ where $\beta$ occurs before $\alpha$ in $\{1,...,\frac{m-1}{2}\}$ then \[(x,y,i)\rightarrow(z,w,l)\] implies $z=xy$, and \[(z,w,l)\rightarrow(u,v,i+\alpha)\] implies $u=zw=xyw$ so that $w=u(xy)^{-1}$. There is one such vertex per choice of such $\beta$. If we suppose $l=i+\beta$ where $\beta$ occurs after $\alpha$ in $\{1,...,\frac{m-1}{2}\}$ then \[(x,y,i)\rightarrow(z,w,l)\] implies $z=xy$ and \[(z,w,l)\rightarrow(u,v,i+\alpha)\] implies $v=zw=zxy$ so that $w=v(xy)^{-1}$. Thus there is one vertex per choice of such $\beta$. Finally, if we suppose $l=i-\beta$ for some $\beta\in\{1,...,\frac{m-1}{2}\}$ then \[(x,y,i)\rightarrow(z,w,l)\] implies $w=xy$, and \[(z,w,l)\rightarrow(u,v,i+\alpha)\] implies $v=zw=zxy$ so that $z=v(xy)^{-1}$ if $\alpha+\beta$ occurs in $\{\frac{m-1}{2}+1,...,m-1\}$, and implies $u=zw=zxy$ so that $z=u(xy)^{-1}$ if $\alpha+\beta$ occurs in $\{1,...,\frac{m-1}{2}\}$. Either way there is one vertex per choice of such $\beta$. Counting over all such vertices we have $2+m-1$ in total. Thus $\mu=m+1$.
\end{proof}
We will eliminate similar cases in the proof of the following theorem as well and, due to its standard counting arguments being not so different from the proof of the previous Theorem, we show less of it.
\begin{theorem}\label{th3.1} (Construction II) Let $G$ be a multiplicative group of order $n\geq 2$, and $g\in G$ be a non-identity, non-idempotent element. Let $\Gamma$ be the digraph with vertex set $\mathcal{V}=G\times G\times \mathbb{Z}_{m}$ where $m\geq 3$ is odd, and adjacency given by $(x,y,i)\rightarrow (u,v,j)$ if and only if one of the following hold:\begin{enumerate}
\item $x=u,i=j$ and $y\ne v$,
\item $y=v,i=j$ and $x\ne u$,
\item $u=xy$ and $j\in \{i+1,...,i+\frac{m-1}{2}\}$,
\item $v=xy$ and $j\in \{i-\frac{m-1}{2},...,i-1\}$,
\item $u=xyg$ and $j\in \{i+1,...,i+\frac{m-1}{2}\}$,
\item $v=xyg$ and $j\in \{i-\frac{m-1}{2},...,i-1\}$.
\end{enumerate} Then $\Gamma$ is a directed strongly regular graph with parameters \\ $(v,k,t,\lambda,\mu)=(mn^{2},2mn-2,2n+4m-6,n+4m-6,6+4(m-2))$.
\end{theorem}
\begin{proof} Clearly $v=mn^{2}$.
\\[1\baselineskip]
$\underline{k}$: Let $(x,y,i)\in \mathcal{V}$. There are $n-1$ vertices $(x,y,i)\in \mathcal{V}$ with $z\ne x$, and $n-1$ vertices $(x,z,i)\in \mathcal{V}$ with $z\ne y$. Also, whenever $\alpha\in\{1,...,\frac{m-1}{2}\}$, there are: $n$ vertices of the form $(xy,z,i+\alpha)$, $n$ vertices of the form $(z,xy,i-\alpha)$, $n$ vertices of the form $(gxy,z,i+\alpha)$, and $n$ vertices of the form $(z,gxy,i-\alpha)$. Thus $k=2nm-2$.
\\[1\baselineskip]
$\underline{t}$: We need to count the number of vertices $(u,v,j)\in \mathcal{V}$ such that \[(x,y,i)\leftrightarrow (u,v,j).\] Let $(u,v,j)\in \mathcal{V}$. If $x=u$ and $i=j$ but $y\ne v$, or if $y=v$ and $i=j$ but $x\ne u$ then $(u,v,j)$ is such a vertex. This gives $2(n-1)$ such vertices.
\par
Let $\alpha\in \{1,...,\frac{m-1}{2}\}$ be fixed. Suppose $(x,y,i)\leftrightarrow (xy,z,i+\alpha)$.  Then either $y=zxy$ so that $z=x^{-1}$, or $y=gzxy$ so that $z=(gx)^{-1}$. Thus there are two such vertices.
\par
If $(x,y,i)\leftrightarrow (z,xy,i-\alpha)$, then we must have either $x=zxy$ so that $z=y^{-1}$, or $x=zgy$ so that $z=(gy)^{-1}$. There are two such vertices.
\par
If $(x,y,i)\leftrightarrow (gxy,z,i+\alpha)$, then either $y=gxyz$ so $z=(gx)^{-1}$, or $g^{2}zxy$ so $z=(g^{2}x)^{-1}$. Since $g^{2}$ is neither $Id$ nor $g$, there are two such vertices.
\par
If $(x,y,i)\leftrightarrow (z,gxy,i-\alpha)$, then either $x=zgxy$ so $x=(gy)^{-1}$, or $x=g^{2}zxy$ so $z=(g^{2})^{-1}$. Again, because of the condition on $g$, there are two such vertices.
\par
Counting over all such vertices (letting $\alpha$ run over $\{1,...,\frac{m-1}{2}\}$) we get $t=2n+4m-6$.
\\[1\baselineskip]
$\underline{\lambda}$: Now, given two vertices $(x,y,i),(u,v,j)\in \mathcal{V}$ such that $(x,y,i)\rightarrow(u,v,j)$, we need to count the number of vertices $(z,w,l)\in \mathcal{V}$ such that \[(x,y,i)\rightarrow(z,w,l)\text{ and }(z,w,l)\rightarrow(u,v,j).\] The six possibilities for such a vertex $(u,v,j)$ are listed in the statement of the theorem.
\par
We only show the first case, where $u=x$, $i=j$ and $y\ne v$, which uses the same standard counting arguments as the other five cases. So we have $(x,y,i)\rightarrow(x,v,i)$ and $y\ne v$. There are $n-2$ vertices $(z,w,l)\in \mathcal{V}$ such that $(x,y,i)\rightarrow(z,w,l)\text{ and }(z,w,l)\rightarrow(u,v,j)$ since, for $z=x$, there are $n-2$ points $w\in G$ such that $y\ne w\ne v$.
\par
There are $\frac{m-1}{2}$ vertices $(z,w,l)\in \mathcal{V}$ where $l=i+\alpha$ for some fixed $\alpha\in\{1,...,\frac{m-1}{2}\}$. Then one of the following must be hold: $z=xy$ and $v=xyw$ so that $w=v(xy)^{-1}$, or $z=gxy$ and $v=gxyw$ so that $w=v(gxy)^{-1}$, or $z=xy$ and $v=gxyw$ so that $w=v(g^{2}xy)^{-1}$.
\par
Also there are $\frac{m-1}{2}$ vertices $(z,w,l)\in \mathcal{V}$ where $l=i-\alpha$ for some fixed $\alpha\in\{1,...,\frac{m-1}{2}\}$. Then one of the following must hold: $w=xy$ and $x=u=xyz$ so that $z=y^{-1}$, or $w=xyg$ and $x=u=gxyz$ so that $z=(gy)^{-1}$, or $w=gxy$ and $x=u=xyz$ so that $z=y^{-1}$. Thus there are $n+4m-6$ such vertices. As the other five cases are similar, we have $\lambda=n+4m-6$.
\\[1\baselineskip]
$\underline{\mu}$: Now, given two vertices $(x,y,i),(u,v,j)\in \mathcal{V}$ such that \[(x,y,i)\nrightarrow(u,v,j),\] we need to count the number of vertices $(z,w,l)\in \mathcal{V}$ such that \[(x,y,i)\rightarrow(z,w,l)\text{ and }(z,w,l)\rightarrow(u,v,j).\]There are only two possibilities for $(u,v,j)$. One case is when $xy\ne u\ne gxy$ and $j=i+\alpha$ for some $\alpha\in \{1,...,\frac{m-1}{2}\}$, and one case is when $xy\ne v\ne gxy$ and $j=i-\alpha$ for some $\alpha\in\{1,...,\frac{m-1}{2}$. We show only the first case, as the second case is similar.
\par
Suppose $(x,y,i)\nrightarrow(u,v,i+\alpha)$ for some fixed $\alpha\in\{1,...,\frac{m-1}{2}\}$, and $xy\ne u\ne gxy$. There are $m$ possibilities for a vertex $(z,w,l)\in\mathcal{V}$: one case is when $l=i$, $\frac{m-1}{2}$ cases are when $l=i+\beta$ for some $\beta\in\{1,...,\frac{m-1}{2}\}$, and another $\frac{m-1}{2}$ cases are when $l=i-\beta$ for some $\beta\in\{1,...,\frac{m-1}{2}\}$.
\par
Suppose $l=i$. If $(x,y,i)\rightarrow(z,w,l)$ then either $x=z$ and $y\ne w$, or $x\ne z$ and $y=w$. If $(z,w,l)\rightarrow(u,v,i+\alpha)$ we must have $u=zw$ or $u=gzw$ if $x=z$ so that $w=ux^{-1}$ or $w=u(gx)^{-1}$, or $u=zy$ or $u=gzy$ if $w=y$, so that $z=uy^{-1}$ or $z=u(gy)^{-1}$. Thus there are four such vertices.
\par
Now suppose $l=i+\alpha$ (i.e. $\beta=\alpha$). If $(z,w,l)\rightarrow(u,v,i+\alpha)$ then we must have either $z=u$ and $w\ne v$, or $z\ne u$ and $w=v$. If $(x,y,i)\rightarrow(z,w,l)$ then we must have $z=xy$ or $z=gxy$. Since $xy\ne u\ne gxy$ we cannot have $z=u$. Thus $z\ne u$ and $w=v$. Then we must have $w=v$ and $z=xy\ne u$, or $z=gxy\ne u$. Thus there are two such vertices. The remaining subcases are: $l=i+\beta$ where $\beta$ occurs before $\alpha$ in the natural order of $\{1,...,\frac{m-1}{2}\}$, $l=i+\beta$ where $\beta$ occurs after $\alpha$ in the natural order of $\{1,...,\frac{m-1}{2}\}$, and $l=i-\beta$ for any $\beta\in \{1,...,\frac{m-1}{2}\}$. It is not difficult to show that there are four vertices per such choice of $\beta$. As the second case, where $xy\ne v\ne gxy$ and $j=i-\alpha$ for some $\alpha\in\{1,...,\frac{m-1}{2}\}$, is similar to the one we have just show, we can conclude that $\mu=6+4(m-2)$. This completes the proof.
\end{proof}
The following table shows new parameters for directed strongly regular graphs constructed in the above way with less than $100$ vertices.

\begin{center}
\begin{longtable}{c|cc|ccccc}\caption{Parameters of new directed strongly regular graphs with less than 100 vertices.}\\
Ref&$m$&$n$&$v$&$k$&$t$&$\lambda$&$\mu$ \\ \hline
Theorem \ref{th3.0}&5&4&80&22&10&6&6 \\
Theorem \ref{th3.1}&9&3&81&52&36&33&34 \\
Theorem \ref{th3.0}&9&3&81&28&12&9&10 \\
Theorem \ref{th3.0}&11&3&99&34&14&11&12 \\
Theorem \ref{th3.1}&11&3&99&64&44&41&42 \\
\end{longtable}\end{center}

Next we discuss how to construct new Cayley graphs that are directed strongly regular.
\section{Constructions of Cayley Graphs that are Directed Strongly Regular}

Olmez and Song in \cite{OS} showed constructed directed strongly regular graphs with parameters $((m+1)s,ls,ld,ld-d,ld)$ where $d$, $l$ and $s$ are integers with $dm=ls$ and $1\leq l<m$. We show how to obtain directed strongly regular graphs with these parameters which are also Cayley graphs. Interestingly, it is possible to construct Cayley graphs which are directed strongly regular with these parameters.
\par
Let $S$ be a subset of a finite group $G$ such that $\langle S \rangle=G$, i.e., $G$ is generated by the members of $S$. The {\it Cayley Graph} of $G$ with respect to $S$, denoted $Cay(G,S)$, is the directed graph with elements of $G$ as its vertices, where $g\rightarrow h$ if and only if $g^{-1}h\in S$.
\par
For any finite group $G$ the {\it group ring} $\mathbb{Z}\left[G\right]$ is defined as the set of all formal sums of elements of $G$, with coefficients in $\mathbb{Z}$. The operations ``$+$'' and ``$\cdot$'' on $\mathbb{Z}\left[G\right]$ are given by \[
\sum_{g\in G}a_{g}g+\sum_{g\in G}b_{g}g=\sum_{g\in G}(a_{g}+b_{g}) \] and \[
\left(\sum_{g\in G}a_{g}g\right)\left(\sum_{h\in G}b_{h}h\right)=\sum_{g,h\in G}gh.\] where are $a_{g},b_{g}\in \mathbb{Z}$.
\par
The group ring $\mathbb{Z}\left[G\right]$ is a ring with multiplicative identity $\mathbf{1}=x^{Id}$, and for any subset $X\subset G$, we denote by $\underline{X}$ the sum $\sum_{x\in X}x$, and we denote by $\underline{X}^{-1}$ the sum $\sum_{x\in X}x^{-1}$.
We will need the following lemmas.
\begin{lemma} (\cite{DUV1}) Let $G$ be a finite group and $S \subset G$ such that $G=\langle S\rangle$. The number of paths of length two from $g$ to $h$ in $Cay(G,S)$ is given by the coefficient of $g^{-1}h$ in $\underline{S}^{2}$.
\end{lemma}
\begin{lemma}\label{le4.0} (\cite{DUV1}) Let $G$ be a finite group and $S \subset G$ such that $G=\langle S\rangle$ and $Id\notin S$. A Cayley graph $Cay(G,S)$ is a directed strongly regular graph with parameters $(v,k,t,\lambda\mu)$ if and only if $|G|=n,|S|=k$ and \[\underline{S}^{2}=t\mathbf{1}+\lambda \underline{S}+\mu (\underline{G}-\mathbf{1}-\underline{S}).\]
\end{lemma}

Let $\mu:H\rightarrow$Aut$K$ be an action of a group $H$ on another group $K$. Let $K\times_{\mu}H$ be the direct product set of $K$ and $H$ with operation \[
(h,k)(h',k')=(h\left[k^{\mu}(h')\right],kk').\] Then $K\times_{\mu}H$ forms a group of order $|K||H|$ with identity $(Id_{K},Id_{H})$ and is called the {\it semidirect product} of $K$ and $H$ with respect to $\mu$. When $\mu$ is clear from the context we simply use $K\rtimes H$ to denote this structure.
\par
Now let $q$ be a prime power and $D$ a multiplicative subgroup of $\mathbb{F}_{q}^{*}$ of order $l$ with generator $\gamma$. Let \[ H=\langle \tau\mid \tau^{l}=1,\tau x=x^{\gamma}\tau\text{ for all }x\in\mathbb{F}_{q}\rangle\] and $G=\mathbb{F}_{q}\rtimes H$. Let $D'\subset\mathbb{F}_{q}$ be $D$-invariant (i.e $xD'=D'$ for all $x\in D$). Set \[
S=-u+D=\{-u+x\mid x\in D\}\] for some nonzero $u\in \mathbb{F}_{q}$, and \[
E=SH=\{xh\mid x\in S,h\in H\}.\] Then we have \[\underline{E}^{2}=\underline{SH}\underline{SH}=\underline{S\Omega}\underline{H},\] where \[S\Omega=\{x+\omega\mid x\in S,\omega\in -uD+D'\}\] (here we used the fact that $xD'=D'$ for all $x\in D$). Note that if \[\underline{D'}\underline{uD}^{-1}=\alpha\mathbf{1}+\beta\mathbb{F}_{q}\] for some $\alpha,\beta\in \mathbb{Z}$. Then we have \[
\underline{E}^{2}=\underline{S\Omega}\underline{H}=\underline{S}(\alpha\mathbf{1}+\beta\underline{\mathbb{F_{q}}})\underline{H}=\alpha\underline{E}
+\beta|S|\underline{G}, \] and, if $0\notin S$, then we have that $Cay(G,E)$ is directed strongly regular. By Lemma \ref{le4.0} we can see that the parameters of $Cay(G,E)$ are given by $(lq,ls,\beta s,\beta s+\alpha,\beta s)$ where $s=|S|$. We have thus shown the following.
\begin{lemma}\label{le4.1} Let $q$ be a prime power and $D$ a multiplicative subgroup of $\mathbb{F}_{q}^{*}$ of order $l$ and generator $\gamma$. Let $G=\mathbb{F}_{q}^{*}\rtimes H$ where $H=\langle \tau\mid \tau^{l}=1,\tau x=x^{\gamma}\tau\text{ for all }x\in\mathbb{F}_{q}\rangle$. Suppose $D'\subset \mathbb{F}_{q}^{*}$ is $D$-invariant, and let $E=SH$ where $S=-u+D'$ for some $u\in\mathbb{F}_{q}^{*}$. If $0\notin S$ and $\underline{D'}\underline{uD}^{-1}=\alpha\mathbf{1}+\beta\mathbb{F}_{q}$ for some $\alpha,\beta\in \mathbb{Z}$, then we have $Cay(G,E)$ is a directed strongly regular graph with parameters $(lq,ls,\beta s,\beta s+\alpha,\beta s)$.
\end{lemma}
We have yet to find new parameters using this method, but the Cayley property makes it interesting, and there may be certain pairs $D,-u+D'$ which do result in directed strongly regular graphs with new parameters. We will need the following lemmas.

\begin{lemma}\label{le4.2} Let $q=em+1$ be a prime power for some positive integers $e$ and $m$. In the group ring $\mathbb{Z}\left[\mathbb{F}_{q}\right]$ we have \[
\underline{D_{i}^{e}}\underline{D_{j}^{e}}=a_{ij}\mathbf{1}+\sum_{k=0}^{e-1}(j-i,k-i)_{e}\underline{D_{k}^{e}}\] where \[
a_{ij}=\begin{cases} m, \text{ if } m \text{ is even and } j=i, \\
                     m, \text{ if } m \text{ is odd and } j=i+\frac{e}{2}, \\
                     0, \text{ otherwise. }\end{cases}\]
\end{lemma}

\begin{lemma}\label{le4.3} (\cite{D}) If $q \equiv 1$ $($mod $4)$ then the cyclotomic numbers of order two are given by
\begin{eqnarray*}
(0,0) & = & \frac{q-5}{4},                         \\
(0,1) & = & (1,0) = (1,1) = \frac{q-1}{4}.
\end{eqnarray*} If $q \equiv 3 ($mod $4)$ then are given by
\begin{eqnarray*}
(0,1) & = & \frac{q+1}{4},                         \\
(0,1) & = & (1,0) = (1,1) = \frac{q-3}{4}.
\end{eqnarray*}
\end{lemma}

We have the following construction of directed strongly regular Cayley graphs.

\begin{theorem} Let $q$ be a prime power, $\gamma$ be a primitive element and write $l=\frac{q-1}{2}$. Let $H=\langle \tau\mid \tau^{l}=1,\tau x=x^{\gamma^{2}}\tau\text{ for all }x\in\mathbb{F}_{q}\rangle$, $G=\mathbb{F}_{q}\rtimes H$, and $E=SH$ where $S\subset\mathbb{F}_{q}$. If $q\equiv 3($mod $4)$ and $S=-u+D_{0}^{2}\cup\{0\}$ where $u$ is a quadratic nonresidue, then $Cay(G,E)$ is a directed strongly regular graph with parameters $(q\frac{q-1}{2},\frac{q+1}{2}\frac{q-1}{2},\frac{q+1}{4}\frac{q-1}{2},\frac{q+1}{4}\frac{q-3}{2},\frac{q+1}{4}\frac{q-1}{2})$. If $q\equiv 1($mod $4)$ and $S=-u+D_{0}^{2}$ where $u$ is a quadratic nonresidue, then $Cay(G,E)$ is a directed regular graph with parameters $(q\frac{q-1}{2},(\frac{q-1}{2})^{2},\frac{q-1}{4}\frac{q-1}{2},\frac{q-1}{4}\frac{q-3}{2},\frac{q-1}{4}\frac{q-1}{2})$.
\end{theorem}
\begin{proof} Let $q\equiv 3($mod $4)$ and $S=-u+D_{0}^{2}\cup\{0\}$ where $u$ is a quadratic nonresidue. Combining Lemmas \ref{le4.2} and \ref{le4.3} we have \[
\underline{(D_{1}^{2})^{-1}}\underline{D_{0}^{2}}=\underline{D_{0}^{2}}(\underline{D_{0}^{2}}+\mathbf{1})=((0,0)_{2}+1)
\underline{D_{0}^{2}}+(0,1)_{2}\underline{D_{1}^{2}},\] where $(0,1)_{2}=(0,0)_{2}+1=\frac{q+1}{4}$. Thus, by Lemma \ref{le4.1} we have \[
\underline{E}^{2}=-\frac{q+1}{4}\underline{E}+\frac{q+1}{4}(\frac{q-1}{2})\underline{G},\] and $Cay(G,E)$ is a directed strongly regular graph with parameters \\$(q\frac{q-1}{2},\frac{q+1}{2}\frac{q-1}{2},\frac{q+1}{4}\frac{q-1}{2},\frac{q+1}{4}\frac{q-3}{2},\frac{q+1}{4}\frac{q-1}{2})$. Similarly, when $q\equiv 1($mod $4)$ and $S=-u+D_{0}^{2}$ where $u$ is a quadratic nonresidue, we have $Cay(G,E)$ is a directed regular graph with parameters \\$(q\frac{q-1}{2},(\frac{q-1}{2})^{2},\frac{q-1}{4}\frac{q-1}{2},\frac{q-1}{4}\frac{q-3}{2},\frac{q-1}{4}\frac{q-1}{2})$.
\end{proof}

We will also make use of the following lemma.
\begin{lemma}\label{le4.4} (\cite{D}) Let $q=4f+1=x^{2}+4y^{2}$ be a prime power with $x,y \in \mathbb{Z}$ and $x \equiv 1$ (mod 4) (here, $y$ is two-valued depending on the choice of the primitive root $\alpha$ defining the cyclotomic classes). The five distinct cyclotomic numbers of order four for even $f$ are
\begin{eqnarray*}
(0,0) & = & \frac{q-11-6x}{16},    \\
(0,1) & = & (1,0) = (3,3) = \frac{q-3+2x+8y}{16}, \\
(0,2) & = & (2,0) = (2,2) = \frac{q-3+2x}{16}, \\
(0,3) & = & (3,0) = (1,1) = \frac{q-3+2x-8y}{16}, \\
\text{all others} & = & \frac{q+1-2x}{16}.
\end{eqnarray*}
\end{lemma}
We first learned of the following from Feng Tao in 2015, but we are unaware of any publication.

\begin{theorem} Let $q=1+16a^{2}$ be a prime power where $a$ is an integer, $\gamma$ a primitive element and $l=\frac{q-1}{4}$. Let $H=\langle \tau\mid \tau^{l}=1,\tau x=x^{\gamma^{4}}\tau\text{ for all }x\in\mathbb{F}_{q}\rangle$, $G=\mathbb{F}_{q}\rtimes H$, and $E=SH$ where $S=-1+D_{2}^{2}\cup\{0\}$. Then $Cay(G,E)$ is a directed strongly regular graph with parameters $(q\frac{q-1}{4},(\frac{q-1}{4})^{2},(\frac{q-1}{8})^{2},(\frac{q-1}{8})^{2}-\frac{q-1}{16},(\frac{q-1}{8})^{2})$.
\end{theorem}
\begin{proof} Combining Lemmas \ref{le4.2} and \ref{le4.4} we have \[
\underline{(D_{0}^{4})^{-1}}\underline{D_{2}^{4}}=\underline{D_{0}^{4}}\underline{D_{2}^{4}}=(2,0)_{4}\underline{D_{0}^{4}}+
(2,1)_{4}\underline{D_{1}^{4}}+(2,2)_{4}\underline{D_{2}^{4}}+(2,3)_{4}\underline{D_{3}^{4}} \] where all of the cyclotomic numbers are equal. Thus, by Lemma \ref{le4.1} we have \[
\underline{E}^{2}=-\frac{q-1}{8}\underline{E}+\frac{q-1}{8}(\frac{q-1}{4})\underline{G},\] and $Cay(G,E)$ is a directed strongly regular graph with parameters \\$(q\frac{q-1}{4},(\frac{q-1}{4})^{2},(\frac{q-1}{8})^{2},(\frac{q-1}{8})^{2}-\frac{q-1}{16},(\frac{q-1}{8})^{2})$.
\end{proof}
We omit the proof of the following as it is similar to the others. The cyclotomic numbers of order six can be found in \cite{STO}.
\begin{theorem} Let $q=1+3a^{2}$ be a prime power such that $2$ is a cubic residue and $a$ is an integer, $\gamma$ a primitive element and $l=\frac{q-1}{6}$. Let $H=\langle \tau\mid \tau^{l}=1,\tau x=x^{\gamma^{6}}\tau\text{ for all }x\in\mathbb{F}_{q}\rangle$, $G=\mathbb{F}_{q}\rtimes H$, and $E=SH$ where $S=-1+D_{3}^{6}$. Then $Cay(G,E)$ is a directed strongly regular graph with parameters $(q\frac{q-1}{6},(\frac{q-1}{6})^{2},\frac{(q-1)^{2}}{216},\frac{(q-1)^{2}}{216}-\frac{q-1}{36},\frac{(q-1)^{2}}{216})$.
\end{theorem}
\section{Constructions of Directed Strongly Regular Graphs from Block Matrices}\label{sec3}

Block matrices were used in \cite{AGOS} and \cite{DUV2} to construct directed strongly regular graphs with parameters $(2(2l+1),2l,l,l-1,l)$ where $l$ is a positive integer. In \cite{DUV1} it was shown how to construct directed strongly regular graphs with parameters $(Nq,N-1,\frac{N-1}{q},\frac{N-1}{q}-1,\frac{N-1}{q})$. In this section we use block matrices to construct directed strongly regular graphs with such parameters. Interestingly, the method requires a certain type of matrix which resembles the multiplication table of a loop.

\par
For a matrix $M=(M_{ij})$ we denote the $ij$th position by $M_{ij}$, and for a block matrix $U=\left[U_{ij}\right]$ we denote the $ij$th block by $\left[U\right]_{ij}$.
\par
We say $M$ is a $\sigma$-circulant if $M_{ij}=M_{i-k,j-\sigma k}$ for all indices $(i,j)$. Note that the sum of two $\sigma$-circulants is again a $\sigma$-circulant, and that the product of a $\sigma_{1}$-circulant with a $\sigma_{2}$-circulant is a $\sigma_{1} \sigma_{2}$-circulant.
\par
We say a matrix $M$ is {\it transpotent} if $M_{il}*M_{lj}=M_{ij}$ for all pairs of indices $(i,l)$ and $(l,j)$. Again, let $q$ be a prime power, $\mathbb{F}_{q}$ a finite field, and $e$ a divisor of $q-1$. Let $M$ be an $e \times e$ transpotent matrix over $\mathbb{F}_{q}$. We say $M$ is a {\it (q,e)-loop matrix} (or simply a {\it (q,e)-loop}) if each row and column of $M$ contains exactly one member of $D_{l}^{e}$ for each $l \in \{0,...,e-1\}$.
\par
Let $M$ be any matrix over $\mathbb{F}_{q}$. For $\sigma \in \mathbb{F}_{q}$ and $D$ any union of cyclotomic cosets of order $e$, let $C_{D}^{(q,e)}$ resp. $C_{\sigma}^{(q,e)}$ denote the matrices, with rows and columns indexed by $\mathbb{F}_{q}$, given by \[
(C_{D}^{(q,e)})_{ij}=\begin{cases} 1 \text{ if } x_{i} \in x_{j}+D, \\
                           0 \text{ otherwise, } \end{cases} \]
resp.
\[(C_{\sigma D}^{(q,e)})_{ij}=\begin{cases} 1 \text{ if } x_{i} \in \sigma x_{j}+D, \\
                           0 \text{ otherwise, } \end{cases} \]
and let $(C_{D}^{(q,e)})^{M}$ denote the block matrix given by $\left[C_{M_{ij} D}^{(q,e)}\right]$. When $D=D_{0}^{e}$ we simply write $C_{1}^{(q,e)}$ and $C_{\sigma}^{(q,e)}$ for $C_{D}^{(q,e)}$ and $C_{\sigma D}^{(q,e)}$ respectively.

\begin{lemma}\label{le2.0} Let $q$ be a prime power with $q-1=em$ for some integer $m$. Let $\sum_{i=0}^{e-1}A_{i}B_{i}$ be a sum of matrix products where all of the $A_{i}$'s and $B_{i}$'s are $0$-$1$ matrices whose rows and columns are indexed by $\mathbb{F}_{q}$, each having the same first row with $1$'s only positions corresponding to $D_{0}^{e}$. Assume that $A_{i}B_{i}$ is $\sigma$-circulant for all $i,0 \leq i \leq e-1$, and also that $B_{i}$ is $\sigma_{i}$-circulant, where $\sigma_{i} \in D_{i}^{e}$. Then, if $C$ is the $\sigma$-circulant matrix whose first row is the same as that of the $A_{i}$'s and $B_{i}$'s, $\sum_{i=0}^{e-1}A_{i}B_{i}+C=mJ$.
\end{lemma}

\begin{proof} Since $\sum_{i=0}^{e-1}A_{i}B_{i}+C$ is a $\sigma$-circulant, we need only show that the first row is $(m,...,m)$. First notice that, for fixed $y$, we have $(A_{i}B_{i})_{0,y}$ is equal to the number of $x \in D_{0}^{e}$ such that $y \in x+\sigma_{i}D_{0}^{e}=x+D_{i}^{e}$ for $i,0\leq i \leq e-1$, and $(C)_{0,y}$ is equal to $1$ if $y \in D_{0}^{e}$ and equal to $0$ if $y \not\in D_{0}^{e}$. It is now easy to see that $\sum_{i=0}^{e-1}\left[(A_{i}B_{i})_{0,y}+(C)_{0,y}\right]=m$.
\end{proof}

\begin{theorem}\label{th2.1} Let $q$ be a prime power with $q-1=em$ for some integer $m$, and $M$ a $(q,e)$-loop. Then the block matrix $(C_{1}^{(q,e)})^{M}$ is the adjacency matrix of a directed strongly regular graph with parameters $(e(em +1),em,m,m-1,m)$.
\end{theorem}

\begin{proof} We have $\left[C^{2}+C\right]_{ij}=\sum_{l=1}^{e}C_{M_{il}}C_{M_{lj}}+C_{M_{ij}}$. Note that all of the $C_{ij}$'s have the same first row with $1$'s only in positions corresponding to $D_{0}^{e}$. Since $M$ is transpotent, we have that $C_{M_{il}}C_{M_{lj}}$ is $M_{ij}$-circulant for all $l,1 \leq l \leq e$. Let $D_{b_{i}^{j}}^{e}$ be the cyclotomic coset containing $M_{ij}$. Then, since $M$ is a $(q,e)$-loop, we have that, in the summation $\sum_{l=1}^{e}C_{M_{il}}C_{M_{lj}}$, $b_{l}^{j}$ runs over $\{0,...,e-1\}$ as $l$ runs over $\{1,...,e\}$. Now, with possible relabeling of indices, we can see that $\sum_{l=1}^{e}C_{M_{il}}C_{M_{lj}}+C_{M_{ij}}$ satisfies the conditions of Lemma \ref{le2.0}, hence $\left[C^{2}+C\right]_{ij}=mJ$ for all indices $(i,j)$, and so $C^{2}+C=mJ$.
\end{proof}
We have the following corollaries.
\begin{corollary} Let $q$ be a prime power with $q=em+1$ and $q \equiv 3 ($mod $4)$. Suppose $a,b \in D_{1}^{2}$ are such that $ab=1$. Let \[ M=\left[ \arraycolsep=3.0pt\def\arraystretch{1.0} \begin{array}{cc}
1 & a \\
b & 1 \end{array} \right]. \]  Then $(C_{1}^{(q,2)})^{M}$ is the adjacency matrix of a directed strongly regular graph with parameters \\ $(2(2m+1),2m,m,m-1,m)$.
\end{corollary}
\begin{proof} That $M$ is a $(q,2)$-loop is trivial. The result then follows immediately from Lemma \ref{th2.1}.
\end{proof}

\begin{corollary} Let $q$ be a prime power with $q=3m+1$. Suppose $a \in D_{i}^{2}$ for $i=1$ or $2$. Let \[ M=\left[ \arraycolsep=3.0pt\def\arraystretch{1.0} \begin{array}{ccc}
1            & a^{2} & a      \\
(a^{2})^{-1} & 1     & a^{-1} \\
 a^{-1}      & a     & 1      \end{array} \right]. \]  Then $(C_{1}^{(q,3)})^{M}$ is the adjacency matrix of a directed strongly regular graph with parameters \\ $(3(3m+1),3m,m,m-1,m)$.
\end{corollary}
\begin{proof} We need only show that $M$ is a $(q,3)$-loop. The result then follows from Lemma \ref{th2.1}. Assume, without loss of generality, that $a \in D_{1}^{3}$. Clearly $a^{2} \in D_{2}^{3}$. Then we must have $(a^{2})^{-1} \in D_{1}^{3}$ and $a^{-1}a^{-1}=(a^{2})^{-1}$.
\end{proof}
\begin{corollary} Let $q$ be a prime power with $q=4m+1$ and $-1$ not a quartic residue. Let $a$ and $b$ be the two distinct square roots of $-1$. Let \[ M=\left[ \arraycolsep=3.0pt\def\arraystretch{1.0} \begin{array}{cccc}
1  & -1 & a & b  \\
-1 & 1  & b & a  \\
b  & a  & 1 & -1 \\
a & b & -1 & 1    \end{array} \right]. \] Then $(C_{1}^{(q,4)})^{M}$ is the adjacency matrix of a directed strongly regular graph with parameters \\ $(4(4m+1),4m,m,m-1,m)$.
\end{corollary}
\begin{proof} We need only show that $M$ is a $(q,4)$-loop, and the result will then follow from Lemma \ref{th2.1}. Assume, without loss of generality, that $a \in D_{1}^{4}$ and $b \in D_{3}^{4}$. We have $a^{2}=b^{2}=-1$ with $a=-b$ and $ab=1$.
\end{proof}

\begin{corollary} Let $q$ be a prime power and $e$ and integer with $q=em+1$. Let $a_{0}=1,a_{1},...,a_{e-1}$ be of representatives of the cyclotomic cosets $D_{0}^{e},D_{1}^{e},...,D_{e-1}^{e}$ respectively. Let \[
M=\left[ \arraycolsep=3.0pt\def\arraystretch{1.3} \begin{array}{ccccccc}
1           & a_{1}            & a_{2}           & a_{3}            &  \cdots & a_{e-2}         & a_{e-1}  \\
a_{1}^{-1}  & 1                & a_{1}^{-1}a_{2} & a_{1}^{-1}a_{3}  & \cdots & a_{1}^{-1}a_{e-2}& a_{1}^{-1}a_{e-1}  \\
a_{2}^{-1}  & a_{2}^{-1}a_{1}  & 1               & a_{2}^{-1}a_{3}  & \cdots & a_{2}^{-1}a_{e-2}& a_{2}^{-1}a_{e-1} \\
a_{3}^{-1}  & a_{3}^{-1}a_{1}  & a_{3}^{-1}a_{2} & 1                & \cdots & a_{3}^{-1}a_{e-2}& a_{3}^{-1}a_{e-1} \\
\vdots      & \vdots           & \vdots          & \vdots           &        & \vdots           & \vdots            \\
a_{e-2}^{-1}& a_{e-2}^{-1}a_{1}& a_{e-2}^{-1}a_{2}& a_{e-2}^{-1}a_{3}& \cdots& 1                & a_{e-2}^{-1}a_{e-1} \\
a_{e-1}^{-1}&a_{e-1}^{-1}a_{1} &a_{e-1}^{-1}a_{2}& a_{e-1}^{-1}a_{3} &  \cdots & a_{e-1}^{-1}a_{e-2}&1   \end{array} \right]. \] Then $(C_{1}^{(q,e)})^{M}$ is the adjacency matrix of a directed strongly regular graph with parameters \\ $(e(em+1),em,m,m-1,m)$.
\end{corollary}
\begin{proof} We need only show that $M$ is a $(q,e)$-loop, and the result will then follow from Lemma \ref{th2.1}. We first show that each row, and each column, of $M$ contains exactly one representative of each $D_{l}^{e},0\leq l\leq e-1$. The entry $M_{ij}$ of $M$ is given by $a_{i-1}^{-1}a_{j-1}$. As $j$ runs over $\{1,...,e\}$, $l$ runs over $\{0,...,e-1\}$ where $D_{l}^{e}$ is the cyclotomic coset containing $a_{j-1}$. Thus, as $j$ runs over $\{1,...,e\}$, $l$ runs over $\{0,...,e-1\}$ where $D_{l}^{e}$ is the cyclotomic coset containing $a_{i-1}^{-1}a_{j-1}$ (with $i$ is fixed). Similarly, as $i$ runs over $\{1,...,e\}$, $l$ runs over $\{0,...,e-1\}$ where $D_{l}^{e}$ is the cyclotomic coset containing $a_{i-1}^{-1}a_{j-1}$ (with $j$ fixed).
\par
To see that $M$ is transpotent, simply note that \begin{eqnarray*} M_{il}M_{lj}  & = & a_{i-1}^{-1}a_{l-1}a_{l-1}^{-1}a_{j-1} \\
                                                                                 & = & a_{i-1}^{-1}a_{j-1} \\
                                                                                 & = & M_{ij} \end{eqnarray*} for any indices $(i,l)$ and $(l,j)$. This completes the proof.
\end{proof}
\begin{example} Consider the $(7,2)$-loop given by \[ M=\left[
\arraycolsep=3.0pt\def\arraystretch{1.0} \begin{array}{cc}
1 & 3 \\
5 & 1 \end{array} \right]. \] Then $(C_{1}^{(7,2)})^{M}$ is given by \[ \left[
\arraycolsep=4.0pt\def\arraystretch{1.0} \begin{array}{ccccccc|ccccccc}
0&1&1&0&1&0&0&0&1&1&0&1&0&0 \\
0&0&1&1&0&1&0&1&0&0&0&1&1&0 \\
0&0&0&1&1&0&1&1&1&0&1&0&0&0 \\
1&0&0&0&1&1&0&0&0&0&1&1&0&1 \\
0&1&0&0&0&1&1&1&0&1&0&0&0&1 \\
1&0&1&0&0&0&1&0&0&1&1&0&1&0 \\
1&1&0&1&0&0&0&0&1&0&0&0&1&1 \\ \hline
0&1&1&0&1&0&0&0&1&1&0&1&0&0 \\
1&0&1&0&0&0&1&0&0&1&1&0&1&0 \\
1&0&0&0&1&1&0&0&0&0&1&1&0&1 \\
0&0&1&1&0&1&0&1&0&0&0&1&1&0 \\
1&1&0&1&0&0&0&0&1&0&0&0&1&1 \\
0&1&0&0&0&1&1&1&0&1&0&0&0&1 \\
0&0&0&1&1&0&1&1&1&0&1&0&0&0 \end{array} \right], \]
and is the incidence matrix for the well known directed strongly regular graph with parameters $(14,6,3,2,3)$.
\end{example}

\begin{example} Consider the $(7,3)$-loop given by \[ M=\left[
\arraycolsep=4.0pt\def\arraystretch{1.0} \begin{array}{ccc}
1 & 2 & 3 \\
4 & 1 & 5 \\
5 & 3 & 1  \end{array} \right]. \]Then $(C_{1}^{(7,3)})^{M}$ gives the incidence matrix for a directed strongly regular graph with parameters $(21,6,3,2,3)$.
\end{example}

\begin{example}\label{ex2.0} Consider the $(5,4)$-loop given by \[ M=\left[
\arraycolsep=4.0pt\def\arraystretch{1.0} \begin{array}{cccc}
1 & 4 & 2 & 3 \\
4 & 1 & 3 & 2 \\
3 & 2 & 1 & 4 \\
2 & 3 & 4 & 1  \end{array} \right]. \] Then $(C_{1}^{(5,4)})^{M}$ gives the incidence matrix for a directed strongly regular graph with parameters $(20,4,1,0,1)$.
\end{example}

\begin{remark} The $(5,4)$-loop in Example \ref{ex2.0} is simply the multiplication table for the loop $\mathbb{F}_{5}^{*}$.
\end{remark}

\section{Concluding Remarks}\label{sec6}
We have discussed three new methods for constructing directed strongly regular graphs. The explicit method discussed in Section 3 gives an abundance of infinite families with new parameters. It was previously unknown how to obtain directed strongly regular graphs with such parameters. The simplicity of the method discussed in Section 3 makes one wonder what other modifications could be introduced to obtain even more new parameter sets. The method discussed in Section 4 shows that directed strongly regular graphs with the same parameters as those discussed in \cite{OS} can be constructed as Cayley graphs of certain non-Abelian groups. We believe it should be possible to obtain directed strongly regular graphs with new parameters using the method of Section 4, but to find groups and their respective subsets which make this possible has thus far proven difficult. The authors included Section 5 because of the interesting properties the required matrices posses, and also simply because it may sometimes prove convenient to have a block matrix construction for directed strongly regular graphs with such parameters.

\bibliographystyle{plain}
\bibliography{myref2}

\begin{thebibliography}{10}

\bibitem{AGOS}
F.~Adams, A.~Gendreau, O.~Olmez, and S.~Y. Song.
\newblock Construction of directed strongly regular graphs using block
  matrices.
\newblock {\em arXiv:1311.1164}, 2013.

\bibitem{BOS}
R.~C. Bose.
\newblock A note on fisher's inequality for balanced incomplete block designs.
\newblock {\em Annals of Mathematical Statistics}, pages 619--620, 1949.

\bibitem{BOSE}
R.~C. Bose.
\newblock Strongly regular graphs, partial geometries, and partially balanced
  deisgns.
\newblock {\em Pacific Journal of Mathematics}, 13:389--419, 1963.

\bibitem{BROU}
A.~E. Brouwer and S.~Hobart.
\newblock Parameters of directed strongly regular graphs.
\newblock {\em http://homepages.cwi.nl}.

\bibitem{CAL}
R.~Calderbank and W.~M. Kantor.
\newblock The geometry of two-weight codes.
\newblock {\em Bull. London Math. soc.}, 18:97--122, 1985.

\bibitem{D}
L.E. Dickson.
\newblock Cyclotomy, higher congruences and $\text{W}$aring's problem.
\newblock {\em Amer. J. Math.}, 57:391--424, 1935.

\bibitem{DHM}
C.~Ding, T.~Helleseth, and H.~Martinsen.
\newblock New families of binary sequences with optimal three-level
  autocorrelation.
\newblock {\em IEEE Trans. Inform. Theory}, 47(1):428--433, 2001.

\bibitem{DING}
C.~Ding, A.~Pott, and Q.~Wang.
\newblock Constructions of almost difference sets from finite fields.
\newblock {\em Des Codes Cryptogr.}, 72:581--592, 2014.

\bibitem{DUV2}
A.~Duval.
\newblock A directed graph version of strongly regular graphs.
\newblock {\em Journal of Combinatorial Theorey, (A)}, 47:71--100, 1988.

\bibitem{DUV1}
A.~Duval and D.~Iourinski.
\newblock Semidirect product constructions of directed strongly regular graphs.
\newblock {\em Journal of Combinatorial Theory, (A)}, 104:157--167, 2003.

\bibitem{FX}
T.~Feng and Q.~Xiang.
\newblock Strongly regular graphs from unions of cyclotomic classes.
\newblock {\em Journal of Combinatorial Theorey, (B)}, 102(4):982--995, 2012.

\bibitem{SG}
S.~Gyurki and M.~Klin.
\newblock New rich families of directed strongly regular graphs.
\newblock Presented at Modern Trends in Algebraic Graph Theory, Slovak
  University of Technology in Bratislava, Slovakia, 2014.

\bibitem{MOOR}
E.~Moorhouse.
\newblock {\em Incidence Geometry}, pages 16--17.
\newblock University of Wyoming, 2007.

\bibitem{NOW}
K.~Nowak.
\newblock A survey on almost difference sets.
\newblock {\em arXiv:1409.0114v1}, 2014.

\bibitem{OLMEZ}
O.~Olmez.
\newblock {\em On Highly Regular Digraphs}.
\newblock PhD thesis, Iowa State University, 2012.

\bibitem{OS}
O.~Olmez and S.~Y. Song.
\newblock Some families of directed strongly regular graphs obtained from
  certain finite incidence structures.
\newblock {\em Graphs and Combinatorics}, 30(6):1529--1549, 2014.

\bibitem{STO}
T.~Storer.
\newblock {\em Cyclotomy and Difference Sets}, pages 65--72.
\newblock Markham, Chicago, 1967.

\end{thebibliography}

\end{document}